\documentclass{amsart}
\usepackage{amsmath,amssymb,amsthm,fullpage}

\newtheorem{theorem}{Theorem}
\newtheorem{lemma}[theorem]{Lemma}

\newtheorem{remark}[theorem]{Remark}

\title{A Fourier-analytic Uniqueness Theorem for Lattice-point Enumerators}
\author{António Rocha-Neves}
\email{arochaneves@proton.me}
\date{\today}

\begin{document}
	
	\begin{abstract}
		We consider a bounded set $P \subset \mathbb{R}^d$ and the lattice-point enumerator $L_P(t) = |tP \cap \mathbb{Z}^d|$ for real $t > 0$. We show that if two bounded measurable sets with boundary of measure zero have the same real-parameter lattice-point enumerators for all integer translates, then their indicator functions agree almost everywhere. As a corollary, any convex body is uniquely determined by this data. Our proof is short and Fourier-analytic, where the key device is a periodic point-counting function whose Fourier coefficients recover the Fourier transform of the indicator function on a dense set. This recovers and extends, with a unified argument, the uniqueness results for rational polytopes and symmetric convex bodies established by Royer \cite{RoyerReconstruction, RoyerSymmetric}, whose proofs relied on intricate case-specific geometric constructions.
	\end{abstract}
	
	\maketitle
	
	\section{Introduction and previous work}
	
	Let $P \subset \mathbb{R}^d$ be a bounded set. For each real $t>0$, we define the real-parameter lattice-point enumerator as
	$$
	L_P(t) := |tP \cap \mathbb{Z}^d|.
	$$
	When $P$ is a rational polytope and $t$ is restricted to positive integers, $L_P(t)$ is the classical Ehrhart quasi-polynomial, which is invariant under integer translations. However, this translational invariance fails when $t$ is allowed to vary continuously over all positive real numbers. The resulting family of functions,
	$$
	t \mapsto L_{P+w}(t), \qquad w\in\mathbb{Z}^d,
	$$
	encodes significantly more geometric information about the set $P$.
	
	The uniqueness of sets determined by this family of functions was first studied by Royer. In two extensive preprints \cite{RoyerReconstruction, RoyerSymmetric}, Royer showed that for rational polytopes and symmetric convex bodies, the equality $L_{P+w}(t) = L_{Q+w}(t)$ for all $w\in\mathbb{Z}^d$ and all $t>0$ implies $P=Q$.
	
	Royer's original methods rely on intricate geometric constructions tailored specifically to rational polytopes and symmetric convex bodies. Consequently, the extension of this uniqueness property to arbitrary real polytopes and to all convex bodies remained open, corresponding to Conjecture 18 in \cite{RoyerReconstruction} and Conjecture 9 in \cite{RoyerSymmetric}, respectively.
	
	The purpose of this short note is to prove both conjectures extending the uniqueness result to all convex bodies, while bypassing the combinatorial complexity of previous methods. Inspired by the continuous Fourier-analytic approach to lattice-point enumeration developed in Robins' book \cite{RobinsBook}, we shift the perspective to the Fourier domain. However, while approaches in this area often rely on the Poisson summation formula, applying this directly to the indicator functions of arbitrary convex bodies requires distributional arguments. We avoid this by defining a periodic point-counting function, which allows us to work within $L^1$ and cleanly obtain the following:
	
	\begin{itemize}
		\item If $P,Q\subset\mathbb{R}^d$ are bounded measurable sets with boundaries of Lebesgue measure zero and $L_{P+w}(t) = L_{Q+w}(t)$ for all $w\in\mathbb{Z}^d$ and all $t>0$, then $1_P = 1_Q$ almost everywhere.
		\item If, in addition, $P$ and $Q$ are convex bodies, then $P=Q$.
	\end{itemize}
	
	In particular, we recover and extend Royer's uniqueness theorems to all convex bodies.
	
	In a similar direction, Higashitani, Murai, and Yoshinaga \cite{Higashitani2025} recently studied the reconstruction of rational polytopes from standard Ehrhart quasi-polynomials (where the dilation parameter $t$ is restricted to integers) evaluated over dense sets of translations, a result they note was previously obtained by Alhajjar. In contrast, by allowing $t$ to vary continuously over the positive reals, we are able to strictly restrict our translations to the discrete integer lattice $\mathbb{Z}^d$, while simultaneously extending the uniqueness property to all convex bodies.
	
	\section{Main analytic theorem}
	
	Throughout, for $f\in L^1(\mathbb{R}^d)$ we use the Fourier transform
	$$
	\widehat{f}(\xi) := \int_{\mathbb{R}^d} f(x)\,e^{-2\pi i\langle \xi,x\rangle}\,dx,
	\qquad \xi\in\mathbb{R}^d.
	$$
	If $t>0$ and $P\subset\mathbb{R}^d$ is measurable, the scaling rule
	$
		\widehat{1_{tP}}(\xi) = t^d \,\widehat{1_P}(t\xi)
	$
	follows from the change of variables $x := ty$.
	
	For any bounded measurable set $P\subset\mathbb{R}^d$ and $t \in \mathbb{R}_{>0}$, define $F_{t, P}: \mathbb{R}^d \rightarrow \mathbb{R}$ by
	$$
	F_{t,P}(x) := \sum_{n\in\mathbb{Z}^d} 1_{tP}(n-x)
	=
	|(tP + x)\cap\mathbb{Z}^d|.
	$$
	
	\begin{lemma}\label{lem:Ft-basic}
		Let $P\subset\mathbb{R}^d$ be a bounded measurable set such that $\lambda(\partial P)=0$, and let $t \in \mathbb{R}_{>0}$.
		Then:
		\begin{enumerate}
			\item $F_{t,P}$ is well-defined and $\mathbb{Z}^d$-periodic.
			\item The discontinuity set of $F_{t,P}$ is contained in
			$$
			D_t(P) := \bigcup_{n\in\mathbb{Z}^d} \bigl(n - \partial(tP)\bigr),
			$$
			which is closed and has Lebesgue measure zero.
			\item $F_{t,P}$ is constant on each connected component of $\mathbb{R}^d \setminus D_t(P)$.
		\end{enumerate}
	\end{lemma}
	
	\begin{proof}
		\mbox{}
		\begin{enumerate}
		\item Boundedness of $P$ implies $tP$ is contained in a ball of finite radius, so
		for each fixed $x$ the set $(tP+x)\cap\mathbb{Z}^d$ is finite. Thus the sum
		defining $F_{t,P}(x)$ has finitely many nonzero terms and is well-defined.
		
		For periodicity, if $m\in\mathbb{Z}^d$ we have
		$$
		F_{t,P}(x+m)
		=
		\sum_{n\in\mathbb{Z}^d} 1_{tP}(n-(x+m))
		=
		\sum_{n'\in\mathbb{Z}^d} 1_{tP}(n'-x)
		=
		F_{t,P}(x),
		$$
		where $n':=n-m$.
		
		\item If $F_{t,P}$ is discontinuous at $x$, then there exists $n\in\mathbb{Z}^d$ such that $x\in n-\partial(tP)$. Thus the discontinuity set is contained in
		$$
		D_t(P) := \bigcup_{n\in\mathbb{Z}^d} (n - \partial(tP)).
		$$
		Noting that $P$ is bounded, we see that this is a locally finite union of closed sets, hence $D_t(P)$ is closed.
		Since $\lambda(\partial P)=0$, we have
		$\lambda(\partial(tP))=0$, and since the union is countable we conclude that $\lambda(D_t(P))=0$.
		
		\item On $U := \mathbb{R}^d \setminus D_t(P)$ the function $F_{t,P}$ is continuous
		by definition of $D_t(P)$. Since $F_{t,P}$ takes values in $\mathbb{Z}_{\ge 0}$,
		it is an integer-valued continuous function on each connected component of $U$,
		and therefore must be constant on each such component.
		\end{enumerate}
	\end{proof}
	
	With the periodicity and almost-everywhere continuity of $F_{t,P}$ established, we can now formulate our main analytic result.
	
	\begin{theorem}\label{thm:analytic}
		Let $P,Q\subset\mathbb{R}^d$ be bounded measurable sets such that
		$\lambda(\partial P)=\lambda(\partial Q)=0$.
		Assume
		\begin{equation}\label{eq:hypothesis}
			L_{P+w}(t) = L_{Q+w}(t)
			\quad\text{for all } w\in\mathbb{Z}^d,\ t \in \mathbb{R}_{>0}.
		\end{equation}
		Then $1_P = 1_Q$ almost everywhere.
	\end{theorem}
	
	\begin{proof}
		By the definition of the lattice point enumerator and $F_{t,P}$ we have
		$$
		F_{t,P}(tw)=|(tP+tw) \cap \mathbb{Z}^d|=L_{P+w}(t) \quad\text{for all } w\in\mathbb{Z}^d,\ t>0.
		$$
		
		Therefore, \eqref{eq:hypothesis} is equivalent to
		$$
		F_{t,P}(tw) = F_{t,Q}(tw)
		\quad\text{for all } w\in\mathbb{Z}^d,\ t>0.
		$$
		
		Fix an irrational $t>0$. For any $w,m\in\mathbb{Z}^d$, using periodicity we have
		$$
		F_{t,P}(tw+m) = F_{t,P}(tw) = F_{t,Q}(tw) = F_{t,Q}(tw+m).
		$$
		
		Thus $F_{t,P}$ and $F_{t,Q}$ agree on
		$$
		\Lambda_t := \{tw + m : w,m\in\mathbb{Z}^d\}.
		$$
		
		For irrational $t$, the set $\{tw \bmod \mathbb{Z}^d : w\in\mathbb{Z}^d\}$ is dense
		in the torus $\mathbb{R}^d/\mathbb{Z}^d$, hence $\Lambda_t$ is dense in $\mathbb{R}^d$.
		
		Let $D := D_t(P)\cup D_t(Q)$, where $D_t(P)$ and $D_t(Q)$ are as in Lemma
		\ref{lem:Ft-basic}. Then $\lambda(D)=0$, and on the open set
		$$
		U := \mathbb{R}^d \setminus D
		$$
		both $F_{t,P}$ and $F_{t,Q}$ are constant on each connected component
		of $U$ (again by Lemma~\ref{lem:Ft-basic}). The components of $U$ are
		open and $\Lambda_t$ is dense, so $\Lambda_t$ intersects each component.
		Therefore $F_{t,P}$ and $F_{t,Q}$ take the same constant value on each
		component of $U$, and we conclude
		$$
		F_{t,P}(x) = F_{t,Q}(x) \quad \text{for almost every } x\in\mathbb{R}^d.
		$$
		
		Since $F_{t,P}=F_{t,Q}$ almost everywhere on $\mathbb{R}^d$ and both are integrable 
		on the torus
		and $\mathbb{Z}^d$-periodic, their Fourier coefficients on the torus coincide.
		For all $\xi\in\mathbb{Z}^d$,
		\begin{equation}\label{eq:coeff-equality}
			\int_{[0,1]^d} F_{t,P}(x)\,e^{-2\pi i\langle \xi,x\rangle}\,dx
			=
			\int_{[0,1]^d} F_{t,Q}(x)\,e^{-2\pi i\langle \xi,x\rangle}\,dx.
		\end{equation}
		
		We now compute the left-hand side in terms of $\widehat{1_P}$. First note that
		$F_{t,P}$ is integer-valued and bounded on $[0,1]^d$, thus $F_{t,P}\in L^1([0,1]^d)$ and the integrals converge absolutely.
		
		Expanding $F_{t,P}$ and interchanging sum and integral (because the term $1_{tP}(n-x)=0$ for all but finitely many $n$), we obtain
		\begin{align*}
			\int_{[0,1]^d} F_{t,P}(x)e^{-2\pi i\langle \xi,x\rangle}\,dx
			&=
			\int_{[0,1]^d}
			\sum_{n\in\mathbb{Z}^d} 1_{tP}(n-x)\,e^{-2\pi i\langle \xi,x\rangle}\,dx \\
			&=
			\sum_{n\in\mathbb{Z}^d}
			\int_{[0,1]^d} 1_{tP}(n-x)\,e^{-2\pi i\langle \xi,x\rangle}\,dx \\
			&=
			\sum_{n\in\mathbb{Z}^d}
			\int_{n-[0,1]^d} 1_{tP}(y)e^{2\pi i\langle \xi,y\rangle}\,dy \\
			&=
			\int_{\mathbb{R}^d} 1_{tP}(y)e^{2\pi i\langle \xi,y\rangle}\,dy
			=
			\widehat{1_{tP}}(-\xi),
		\end{align*}
		where we substitute $y:=n-x$ and note that $e^{-2\pi i\langle \xi,x\rangle}=e^{2\pi i\langle \xi,n-x\rangle}$ while $\{n-[0,1]^d : n\in\mathbb{Z}^d\}$ tiles $\mathbb{R}^d$.
		
		By the scaling rule for the Fourier transform,
		$
		\widehat{1_{tP}}(-\xi) = t^d\,\widehat{1_P}(-t\xi)
		$.
		Combining this with \eqref{eq:coeff-equality} yields, for our fixed irrational $t>0$,
		$$
		\widehat{1_P}(-t\xi) = \widehat{1_Q}(-t\xi)
		\quad\text{for all } \xi\in\mathbb{Z}^d.
		$$
		
		Since \eqref{eq:hypothesis} holds for all $t>0$, the same argument applies to every
		irrational $t>0$. Let
		$$
		S := \{t\xi : t\in\mathbb{R}\setminus\mathbb{Q},\,\xi\in\mathbb{Z}^d\}.
		$$
		
		Because the set of rational directions $\xi/||\xi||$ for non-zero $\xi \in \mathbb{Z}^d$ is dense in the unit sphere, and the set of irrational multiples $t$ is dense in $\mathbb{R}_{>0}$, the set $S$ forms a dense collection of rays. We conclude that $S$ is dense in $\mathbb{R}^d$.
		
		
		We have $\widehat{1_P}=\widehat{1_Q}$
		on the dense set $S$. As $1_P,1_Q\in L^1(\mathbb{R}^d)$, their Fourier
		transforms are continuous, so $\widehat{1_P}=\widehat{1_Q}$ on all of $\mathbb{R}^d$.
		By the uniqueness of the Fourier transform on $L^1(\mathbb{R}^d)$, this implies $1_P=1_Q$ almost everywhere.
	\end{proof}
	
	\section{Convex bodies and polytopes}
	
	Theorem \ref{thm:analytic} establishes that the indicator functions of the sets coincide almost everywhere in the sense of Lebesgue measure. To elevate this analytic equivalence to strict set equality, we now restrict our attention to convex bodies. Recall that a convex body is a compact convex set with a nonempty interior.
	
	\begin{theorem}\label{thm:convex}
		Let $P,Q\subset\mathbb{R}^d$ be convex bodies. Assume
		$$
		L_{P+w}(t) = L_{Q+w}(t)
		\quad\text{for all } w\in\mathbb{Z}^d,\ t>0.
		$$
		Then $P=Q$.
	\end{theorem}
	
	\begin{proof}
		Convex bodies are bounded measurable sets whose boundaries have Lebesgue measure zero so we can apply Theorem~\ref{thm:analytic} and conclude that $1_P = 1_Q$ almost everywhere. Therefore, the set difference $P \triangle Q$ has measure zero, i.e.
		$$
		\lambda(P\triangle Q) = 0.
		$$
		
		We first show that $\operatorname{int} P \subset \operatorname{int} Q$.
		Suppose for the sake of contradiction that $\operatorname{int} P \not\subset \operatorname{int} Q$. Recall that the interior of $Q$ is the largest open set contained in $Q$. Since $\operatorname{int} P$ is an open set, if it were entirely contained in $Q$, it would necessarily be a subset of $\operatorname{int} Q$. Thus, our assumption implies that $\operatorname{int} P$ is not entirely contained in $Q$.
		
		Therefore, the intersection $\operatorname{int} P \setminus Q$ is non-empty. Since $\operatorname{int} P$ is open and $Q$ is closed, $\operatorname{int} P \setminus Q$ is a non-empty open set, which means it has strictly positive Lebesgue measure. This implies
		$$
		\lambda(P \setminus Q) \ge \lambda(\operatorname{int} P \setminus Q) > 0,
		$$
		which directly contradicts $\lambda(P\triangle Q) = 0$.
		
		Thus, we must have $\operatorname{int} P \subset \operatorname{int} Q$. By symmetry, $\operatorname{int} Q \subset \operatorname{int} P$, hence $\operatorname{int} P = \operatorname{int} Q$.	
		
		For a convex body $K$, one has $K = \overline{\operatorname{int} K}$. Hence
		$$
		P = \overline{\operatorname{int} P} = \overline{\operatorname{int} Q} = Q,
		$$
		as claimed.
	\end{proof}
	
	\begin{remark} \label{rem:polytopes}
		Polytopes are convex bodies, so Theorem~\ref{thm:convex} applies in particular
		to arbitrary (possibly non-rational) convex polytopes. In the terminology of
		\cite{RoyerReconstruction,RoyerSymmetric}, this settles Conjecture~18 of
		\cite{RoyerReconstruction} (for full-dimensional real polytopes) and
		Conjecture~9 of \cite{RoyerSymmetric} (for convex bodies).
	\end{remark}
	
	We finish by mentioning a related open problem. While Theorem~\ref{thm:convex} requires the equality of the lattice-point enumerators for all integer translates $w \in \mathbb{Z}^d$, a natural question is whether a finite set of carefully chosen translations is sufficient. It remains an open problem whether given polytopes $P$ and $Q$ there exists a finite set of vectors $W \subset \mathbb{Z}^d$ (which may depend on $P$ and $Q$) such that $L_{P+w}(t) = L_{Q+w}(t)$ for all $w \in W$ and all $t > 0$ implies $P = Q$. This corresponds to Conjecture 7.28 in \cite{RobinsBook}. The purely Fourier-analytic methods presented here rely on the dense periodic evaluations provided by the full integer lattice, meaning the finite-translate problem likely requires a more refined approach.
	
	\medskip
	
	\noindent\textbf{Acknowledgments.}
	
	The author would like to thank Sinai Robins for suggesting the study of this problem, as well as many insightful discussions.
	
	\noindent\textbf{AI Disclosure.}
	
	The author made use of an AI assistant during the development of this work.
	
	\bibliographystyle{alpha}
	\bibliography{refs}
	
\end{document}